\documentclass[12pt]{article} 
\usepackage{color}
\def \n{\noindent}
\def \bs{\bigskip}
\def \l{\ell}
\def \th{\theta}
\def \S{\mathcal S}

\newtheorem {theorem}{Theorem}
\newtheorem {corol}{Corollary}
\newtheorem {consec}{Corollary}
\newtheorem {lemma}{Lemma}
\newcommand{\point}{\hspace{-1.75mm}{\bf.\ }}

\newcommand{\btheorem}{\begin{theorem}\point}
\newcommand{\etheorem}{\end{theorem}}
\newcommand{\blemma}{\begin{lemma}\point}
\newcommand{\elemma}{\end{lemma}}
\newcommand{\bpro}{\begin{pro}\point}
\newcommand{\epro}{\end{pro}}
\newcommand{\bcorol}{\begin{corol}\point}
\newcommand{\ecorol}{\end{corol}}
\newcommand{\bnote}{\begin{note}\point}
\newcommand{\enote}{\end{note}}
\newcommand{\bconsec}{\begin{consec}\point}
\newcommand{\econsec}{\end{consec}}
\newcommand{\bdefin}{\begin{defin}\point}
\newcommand{\edefin}{\end{defin}}
\newtheorem {pro}{Proposition}
\newtheorem{conj}{Conjecture}

\newcommand{\eqref}[1]{\ensuremath{(\ref{#1})}}

\begin{document}
\title{Extremal sizes of subspace partitions\thanks{This research project was supported by grant KAW 2005.0098 
from the Knut and Alice Wallenberg Foundation.}
\author{O. Heden, J. Lehmann, E. N\u{a}stase, and P. Sissokho}}
\maketitle
\begin{abstract}
A {\em subspace partition} $\Pi$ of $V=V(n,q)$ is a collection of subspaces 
of $V$ such that each $1$-dimensional subspace of $V$ is in exactly one
subspace of $\Pi$. The {\em size} of $\Pi$ is the number of its subspaces.
Let $\sigma_q(n,t)$ denote the {\em minimum size} of a subspace 
partition of $V$ in which the largest subspace has dimension $t$, and let 
$\rho_q(n,t)$ denote the {\em maximum size} of a subspace 
partition of $V$ in which the smallest subspace has dimension $t$.
In this paper, we determine the values of $\sigma_q(n,t)$ and $\rho_q(n,t)$ 
for all positive integers $n$ and $t$. Furthermore, we prove that if 
$n\geq 2t$, then the minimum size of a maximal partial $t$-spread in $V(n+t-1,q)$
is $\sigma_q(n,t)$. 
\end{abstract}

\section{Introduction}\label{sec:intro}
Let $V=V(n,q)$ denote a vector space of dimension $n$ over a finite field with $q$ elements. A {\em subspace partition} $\Pi$ of $V$ is a collection of subspaces 
of $V$ such that each $1$-dimensional subspace of $V$ is in exactly one
subspace of $\Pi$. A subspace partition $\Pi$ is also called a {\em vector space partition} 
(or simply a {\em partition}) of $V$. 
There is a rich literature about vector space partitions, see e.g. \cite{Be2,BESSSV,Bu,He0,Sp} and the 
references therein.

The {\em size} of $\Pi$ is the number of its subspaces.
Let $\sigma_q(n,t)$ denote the {\em minimum size} of a subspace 
partition of $V$ in which the largest subspace has dimension $t$, and let 
$\rho_q(n,t)$ denote the {\em maximum size} of a subspace 
partition of $V$ in which the smallest subspace has dimension $t$.
The purpose of this study is to find these numbers. Since $\sigma_q(n,n)=\rho_q(n,n)=1$, and $\sigma_q(n,1)=\rho_q(n,1)=(q^{n}-1)/(q-1)$, 
we will focus on the case $1<t<n$. Moreover, if $t$ divides $n$, then $\sigma_q(n,t)=\rho_q(n,t)$ 
is the size of a $t$-spread in $V$, i.e., a subspace partition of $V$ in which all the subspaces have dimension $t$.

We will prove the following theorem:
\begin{theorem}\label{thm:1}  
Let $n, k, t$, and $r$ be integers such that $0\leq r < t$, $k\geq 2$, and
$n=kt+r$. Then
\[\rho_q(n,t)=q^{t+r} \sum\limits_{i=0}^{k-2}q^{it}+1\;,
\]
and if furthermore $1\leq r<t$, then
\[\sigma_q(n,t)=q^{t+r} \sum\limits_{i=0}^{k-2}q^{it}+q^{\lceil \frac{t+r}{2}\rceil}+1\;.
\]
\end{theorem}
This theorem improves a result of Beutelspacher~\cite{Be3} who in 1980 proved that
\[
\sigma_q(n,t)\geq q^{\lceil\frac{n}{2}\rceil}+1.
\]
We must also remark that the last two authors of this paper recently found the value of $\sigma_q(2t+1,t)$, see~\cite{NaSi}. They used some equations for subspace partitions derived by the first two authors in~\cite{HeLe}. Furthermore, our derivation of the value of $\sigma_q(n,t)$ uses arguments quite similar to those used in~\cite{NaSi}. 

After some preliminary results in Section \ref{sec:2}, we will prove our theorem in Section \ref{sec:3} and Section \ref{sec:4}. Finally, in Section \ref{sec:5}, we combine our result on $\sigma_q(n,t)$ with a construction of P. Govaerts~\cite{Go} to show that the minimum size of a maximal partial $t$-spread in $V(n+t-1,q)$ is $\sigma_q(n,t)$ for any integer $n\geq 2t$.

\section{Preliminary Results}\label{sec:2}
Let $\Pi$ be a subspace partition of $V=V(n,q)$, $n\geq2$,
with $m_i$ subspaces of dimension $i$, $1\leq i\leq n-1$.
Let $H$ be any hyperplane, i.e., any $(n-1)$-dimensional subspace of $V$, and let $b_i\leq m_i$ be the number of subspaces 
of $\Pi$ that are contained in $H$. We say that $(m_{n-1},\ldots,m_1)$
is the {\em type} of $\Pi$ and $b=(b_{n-1},\ldots,b_1)$ is the {\em type of the hyperplane} 
$H$ (with respect to $\Pi$). Let $s_b$ denote the number of hyperplanes in 
$V$ of type $b$ and define the set 
\[B=\{b:\; s_b>0\}.\]
For $1\leq i\leq n$, let 
\[\theta _i=\frac{q^i-1}{q-1}\] 
denote the number of $1$-dimensional subspaces in an $i$-space; then 
\[h_q(n,i)=\max\left\{0,\theta_{n-i}\right\}\]
denotes the number of hyperplanes containing a given $i$-dimensional subspace.
The following two lemmas were derived in \cite{HeLe}.
\begin{lemma}\label{HeLe2}
Let $\Pi$ be a subspace partition of $V=V(n,q)$ of type $(m_{n-1},\ldots,m_1)$ and let $b=(b_{n-1},\ldots,b_1)$ 
be the type of the hyperplane $H$ with respect to $\Pi$. Let $s_b$ denote the number of hyperplanes 
in $V$ with type $b$. Assume furthermore that $\Pi$ contains a subspace of dimension $d$ and a subspace of 
dimension $d'$, with $1\leq d,d'\leq n-2$. Then\\

\n $(i)$ $\sum\limits_{b\in B}s_b=\frac{q^{n}-1}{q-1}=h_q(n,0)$,

\n $(ii)$ $\sum\limits_{b\in B}b_ds_b=m_d h_q(n,d)$,

\n $(iii)$ $\sum\limits_{b\in B}{b_d\choose 2}s_b={m_d\choose 2}h_q(n,2d)$,

\n $(iv)$ $\sum\limits_{b\in B}b_db_{d'} s_b=m_d m_{d'}h_q(n,d+d')$.
\end{lemma}

\begin{lemma}\label{HeLe1}
Let $\Pi$ be a subspace partition of $V=V(n,q)$ and let $(b_{n-1},\ldots,b_1)$ be 
the type of the hyperplane $H$ with respect to $\Pi$. Then the number of subspaces in $\Pi$ is 
\[|\Pi|=1+\sum_{i=1}^{n-1}b_iq^{i}.\] 
\end{lemma}
We will also use the following lemma  due to Herzog and Sch\"onheim \cite{HeSc} and independently Beutelspacher~\cite{Be2} and Bu~\cite{Bu}.
\begin{lemma}\label{lem:Be}
Let $n$ and $d$ be integers such that $1\leq d \leq n/2$. Then $V=V(n,q)$ admits 
a partition with one subspace of dimension $n-d$ and $q^{n-d}$ subspaces 
of dimension $d$.
\end{lemma}

For $n=kt+r$, $0\leq r<t$, and $k\geq 2$, let 
\begin{equation}\label{def:l}
\ell=q^r\sum\limits_{i=0}^{k-2}q^{it}.
\end{equation} 
The following proposition is an immediate consequence of Lemma \ref{lem:Be}.
\begin{pro}\label{pro:1}
Let $n, k, t$, and $r$ be integers such that $0\leq r < t$, $k\geq 2$, and
$n=kt+r$. Then $V=V(n,q)$ admits a partition $\Pi_m$ of size
\[
|\Pi_m|=\ell\cdot q^t+1,
\]
consisting of $\ell q^t$ subspaces of dimension $t$ and one subspace of dimension $t+r$.
If furthermore, $1\leq r<t$, then $V$ admits a partition $\Pi_M$ of size
\[
|\Pi_M|=\ell\cdot q^t+q^{\lceil \frac{t+r}{2}\rceil}+1\;,
\] 
consisting of $\ell q^t$ subspaces of dimension $t$, $q^{\lceil(t+r)/2\rceil}$ subspaces of dimension 
$\lfloor (t+r)/2\rfloor$ and one subspace of dimension $\lceil (t+r)/2\rceil$.
\end{pro}

We close this section by giving three relations that  will be frequently used. They follow easily from the definitions of $\ell$ and the function $\theta_i$; the third is an immediate consequence of the first two:
\begin{equation}\label{rel1}
\theta_{n-t}-\theta_r=\l\theta_t,
\end{equation}
\begin{equation}\label{rel2}
\theta_{a+b}-\theta_b=q^b\theta_{a}\;,
\end{equation}
\begin{equation}\label{eq:4}
\theta_n-\ell q^t\theta_t=\theta_{t+r}\;.
\end{equation}

\section{The minimum size}\label{sec:3}

In this section we will find $\sigma_q(n,t)$, as indicated in Theorem \ref{thm:1}.
We will need the following lemma, which may be of independent interest.

\begin{lemma}\label{lem:r=1}  
Let $n, k, t$, and $r$ be integers such that $k \geq 2, 1\leq r<t$, and
$n=kt+r$. Let $\Pi$ be a subspace partition of $V=V(n,q)$ with 
no subspace of dimension higher than $t$. Assume furthermore that $\Pi$ contains 
a subspace of dimension $t$ and a subspace of dimension $d$, with $0\leq d <t$. 
Then
\[|\Pi|\geq q^{t+r}\sum\limits_{i=0}^{k-2}q^{it}+q^{d}+1.\]
\end{lemma}
\begin{proof}
Let $\Pi$ be a subspace partition of $V$ containing subspaces of dimension $t$ and $d$ 
with $t>d$. 
Since there exist subspaces of dimensions $t$ and $d$ in $\Pi$, 
we have $m_t>0$ and $m_d>0$. So it follows from Lemma~\ref{HeLe2}(iv) that 
\begin{eqnarray}\label{eqa5}
\sum_{b\in B}b_tb_ds_b=m_tm_d\theta_{n-t-d}\not=0.
\end{eqnarray}
Additionally,
\begin{eqnarray*}\label{eqa6}
\sum_{b\in B}b_tb_ds_b=\sum_{b\in B\atop 0\leq b_t\leq\l-1}b_tb_ds_b+
\sum_{b\in B\atop b_t\geq \l}b_tb_ds_b.
\end{eqnarray*}
If 
\[\sum\limits_{b\in B,\;b_t\geq \l} b_tb_ds_b\not=0,\]
then there exists $b\in B$ such that  
$b_t\geq \l $, $b_d\geq 1$, and $s_b\geq1$. In this case, Lemma~\ref{HeLe1} yields 
\begin{eqnarray*}\label{eqM}
\quad |\Pi|=\sum_{i=1}^{n-1}b_iq^{i}+1\geq b_tq^{t}+b_dq^{d}+1\geq\l\;q^{t}+q^{d}+1,
\end{eqnarray*}
and the lemma follows. So we may assume that $\sum\limits_{b\in B,\; b_t\geq \l }b_tb_ds_b=0$. 
This assumption, combined with \eqref{eqa5} and Lemma~\ref{HeLe2}(iv), yields  
\begin{eqnarray}\label{eqa7}
(\l-1)m_d\theta_{n-d}
&=&\sum_{b\in B}(\l-1)\cdot b_ds_b\cr
&=&\sum_{b\in B\atop 0\leq b_t\leq \l-1}(\l-1)\cdot b_ds_b+\sum_{b\in B\atop b_t\geq \l}(\l-1)\cdot b_ds_b\cr
&\geq&\sum_{b\in B\atop 0\leq b_t\leq \l-1}b_t\cdot b_ds_b+0\cr
&=&\sum_{b\in B\atop 0\leq b_t\leq \l-1}b_t\cdot b_ds_b+\sum_{b\in B\atop b_t\geq \l }b_t\cdot b_ds_b\cr
&=&\sum_{b\in B}b_tb_ds_b\cr
&=&m_tm_d\theta_{n-t-d}
\end{eqnarray}
Since $m_d>0$, dividing both sides of~\eqref{eqa7} by $m_d$ yields
\[ 
m_t \leq \frac{(\l-1)\;\theta_{n-d}}{\theta_{n-t-d}}.
\]
We now show that this implies that 
\begin{equation}\label{eqa8}
m_t\leq(\ell-1)q^t+q^d\;. 
\end{equation}

From \eqref{rel2} we obtain that $\theta_{n-d}=\theta_t+q^t\theta_{n-d-t}$, and hence it remains to prove that
\[
\frac{(\ell-1)\theta_t}{\theta_{n-d-t}}\leq q^d\;.
\]
This fact follows from Equations~\eqref{rel1},~\eqref{rel2} and~\eqref{eq:4}:
\[
q^d\theta_{n-d-t}-\ell\theta_t+\theta_t=\theta_{n-t}-\theta_d-\theta_{n-t}+\theta_r+\theta_t=\theta_t+\theta_r-\theta_d\;,
\]
as $\theta_t>\theta_d$.

Note that $\Pi$ is the disjoint union of ${\mathcal A}=\{W\in \Pi:\; \dim(W)=t\}$ and 
${\mathcal B}=\{W\in \Pi:\; \dim(W)\leq t-1\}$.  By counting the 1-dimensional subspaces not taken up by
$\mathcal{A}$, we can bound the size of $\mathcal{B}$ by
\[ |\mathcal{B}|\geq \frac{\theta _{n}-|\mathcal{A}|\cdot \theta _{t}}{\theta _{t-1}}\;.\]
Since $|{\mathcal A}|=m_t$, we obtain from~\eqref{eqa8} that
\begin{equation}\label{eqa1}
|\Pi|=|{\mathcal A}|+|{\mathcal B}|
\geq m_t+\frac{\theta _{n}-m_t\cdot \theta _{t}}{\theta _{t-1}}  \geq\frac{\theta_n-(\l\,q^{t}-q^{t}+q^{d})(\theta_{t}-\theta_{t-1})}{\theta_{t-1}}\;.\end{equation}
By using Equation~\eqref{eq:4}, the above inequality can be further simplified
\[
|\Pi|\geq
\ell q^t+q^d+\frac{\theta_{t+r}+q^t(\theta_{t}-\theta_{t-1})-q^d\theta_{t}}{\theta_{t-1}}>\ell q^t+q^d+\frac{q^t(\theta_{t}-\theta_{t-1})-q^d\theta_{t}}{\theta_{t-1}}\;.
\]
As furthermore,
\[
q^t(\theta_{t}-\theta_{t-1})=q^{2t-1}>q^d\theta_t\;,
\]
we finally obtain 
\[|\Pi|\geq \l\,q^{t}+q^{d}+1.\]
This concludes the proof of the lemma.
\end{proof}
\medskip

We now prove that under the assumptions of Theorem~\ref{thm:1}, $\sigma_q(n,t)= \l\,q^{t}+q^{\lceil \frac{t+r}{2}\rceil}+1$.

\medskip
\begin{proof}
Let $\Pi$ be a subspace partition of $V=V(n,q)$ in which the largest subspace has dimension 
$t$. Let $\beta=\lceil (t+r)/2\rceil$. If there is a subspace of dimension $d$ in $\Pi$ with $\beta\leq d<t$, then by Lemma~\ref{lem:r=1} 
\begin{equation}\label{eqb11}
|\Pi|\geq \l\,q^{t}+q^{d}+1\geq  \l\,q^{t}+q^{\beta}+1.
\end{equation}
It remains to consider the case where every subspace in $\Pi$ has either dimension $t$ or a dimension less than or equal to $\beta-1$. 

If there exists a hyperplane $H$ of type $b$ with $b_t \geq \l+1,$ then by Lemma \ref{HeLe1}
\begin{equation}\label{eqb22}
|\Pi|=\sum_{i=1}^{n-1}b_iq^{i}+1\geq(\l+1)q^{t}+1\geq\l\,q^{t}+q^{\beta}+1,
\end{equation}
where the last inequality holds since $\beta\leq t$.

So now assume that if $s_b\neq0$ then $b_t\leq \l$. Then 
Lemma~\ref{HeLe2}(ii) yields
\begin{equation}\label{eqb33}
m_t\theta_{n-t}=\sum_{b\in B} b_ts_b\leq\l\cdot\sum_{b\in B}s_b=\l\cdot\theta_n\;.
\end{equation}

From \eqref{rel1}, we derive $\ell\theta_t<\theta_{n-t}$. By combining this inequality with~\eqref{rel2},
we obtain
\[
\ell\theta_n=\ell(q^t\theta_{n-t}+\theta_t)=\ell q^t\theta_{n-t}+\ell\theta_t<\ell q^t\theta_{n-t}+\theta_{n-t}\;.
\]
Consequently, \eqref{eqb33} yields
\begin{equation}\label{eqb34}
m_t<\ell q^t +1\;.
\end{equation}
Note that $\Pi$ is the disjoint union of ${\mathcal A}=\{W\in \Pi:\; \dim(W)=t\}$ and 
${\mathcal B}=\{W\in \Pi:\; \dim(W)\leq \beta-1\}$. 
By Equation~\eqref{eqb34} and since $m_t$ is an integer, we may assume that $m_t\leq \l\,q^{t}$. So by using 
Equation~\eqref{rel2}, we obtain that 
\[
|\Pi|= |{\mathcal A}|+|{\mathcal B}|\, \geq \, m_t+\frac{\th_{n}-m_t\cdot \th_{t}}{\th_{\beta-1}}=\frac{\th_{n}-m_t(\th_{t}-\th_{\beta-1})}{\th_{\beta-1}}
\geq\frac{\th_{n}-\ell q^t(\th_{t}-\th_{\beta-1})}{\th_{\beta-1}},
\]
and hence from \eqref{eq:4}, and the fact that $\theta_{\beta-1}(q^\beta+1) \leq \theta_{t+r}$, we conclude that
\begin{equation}\label{eqb44}
|\Pi|\geq \ell q^t+\frac{\theta_n-\ell q^t\theta_t}{\theta_{\beta-1}}=\ell q^t+\frac{\theta_{t+r}}{\theta_{\beta-1}}
\geq\ell q^t+q^\beta+1\;.
\end{equation}
Summarizing the distinct cases we have considered, we thus obtain 
\begin{equation}\label{eqb4.11}
\sigma_q(n,t)\geq \l\,q^{t}+q^{\beta}+1.
\end{equation}
Finally, by using Proposition \ref{pro:1} we may conclude that
\[
\sigma_q(n,t)= \l\,q^{t}+q^{\beta}+1.
\] 
\end{proof}
%

\section{The maximum size}\label{sec:4}
In this section we now prove that under the assumptions of Theorem~\ref{thm:1}, $\rho_q(n,t)= \l\,q^{t}+1$.

\medskip
\begin{proof}
Let $\Pi$ be a subspace partition of $V=V(n,q)$ in which the smallest subspace has dimension $t$. Suppose $|\Pi| >\l\,q^{t}+1$. The type of $\Pi$ is $(m_{n-1},\ldots,m_t,0,\ldots,0)$. Let $H$ be any hyperplane
of $V$, and let $(b_{n-1},\ldots,b_t,0,\ldots,0)$ be the type of $H$ with respect to $\Pi$. 
Then by Lemma~\ref{HeLe1}, we have 
\[|\Pi|=1+\sum_{i=t}^{n-1}b_iq^{i}=1+q^t\sum_{i=t}^{n-1}b_iq^{i-t}.\]
Thus, $|\Pi|\equiv 1\pmod{q^t}$, and by our above assumption on $|\Pi|$, we have $|\Pi| \geq \l\,q^{t}+q^t+1$.
As the dimension of each member of $\Pi$ is at least $t$, we may use relation \eqref{eq:4} and the fact that $(q^t+1)\theta_t=\theta_{2t}$ to conclude that
\[
\theta_n \geq |\Pi|\,\theta_t\geq(\ell q^t+q^t+1)\theta_t=\theta_n-\theta_{t+r}+\theta_{2t}\;,
\]
which is a contradiction as $\theta_{2t} > \theta_{t+r}$.  Thus $|\Pi| \leq \l\,q^{t}+1$. Since $\Pi$ is an arbitrary partition,  
we obtain
\begin{equation}\label{eqc.2}
\rho_q(n,t) \leq \l\,q^{t}+1.
\end{equation}
Hence, from Proposition \ref{pro:1} now follows that
\[ \rho_q(n,t) = \ell\,q^{t}+1.\]
\end{proof}
%
\section{Application to maximal partial \protect\boldmath $t$-spreads}\label{sec:5}
A {\em partial $t$-spread} of $V=V(n,q)$ is a collection $\S=\{W_1,\ldots,W_k\}$ 
of $t$-dimensional subspaces of $V$ such that $W_i\cap W_j=\{0\}$ for $i\not=j$. 
The {\em size} of $\S$ is its cardinality $|\S|$.
If $V=\bigcup_{W\in\S}W$, then $\S$ is called a $t$-{\em spread}. 
A  partial $t$-spread is called {\em maximal} if it cannot be extended to a larger one. 
Maximal partial $t$-spreads have been extensively studied, see e.g.~\cite{BBW,EiStSz,GaSz,Go,He1,HoPa,JuSt}.
They can be used to construct error-correcting codes~\cite{ClDu,DrFr}, orthogonal arrays~\cite{CoGe,EJSSS}, 
and recently factorial designs~\cite{RBD}.

We let $\tau_q(n,t)$ denote the {\em minimum number} of subspaces in any maximal partial $t$-spread 
of $V(n,q)$. A maximal partial $t$-spread $\S$ of $V(n,q)$ such that $|\S|=\tau_q(n,t)$, 
is called a {\em minimum size} maximal partial $t$-spread. 
Let $n$ and $t$ be fixed integers and let $k$ and $r$ be the unique integers defined by 
$n=kt+r$ and $0\leq r<t$. Beutelspacher~\cite{Be2} showed that if $r=0$ and $k\geq 2$, then
\[
\tau_q(n+t-1,t)=\sigma_q(n,t)=\frac{q^{kt}-1}{q^t-1}.
\]
For $r>0$, P. Govaerts~\cite{Go} proved several results related to the number $\tau_q(n+t-1,t)$.
In particular, he provided the following upper bound for $\tau_q(n+t-1,t)$.
\begin{lemma}[Govaerts~\cite{Go}]\label{lem:Go}
Let $n$ and $t>1$ be integers such that $n\geq 2t$. 
Then there exist $($see~page 610 in~\cite{Go} for a construction$)$ maximal partial 
$t$-spreads of $V(n+t-1,q)$ of size $\sigma_q(n,t)$. Consequently, $\tau_q(n+t-1,t)\leq \sigma_q(n,t)$.
\end{lemma}

We will prove the following theorem.

\begin{theorem}\label{thm:3}
Let $n$ and $t >1 $ be integers such that $n\geq 2t$. 
Then $\tau_q(n+t-1,t)=\sigma_q(n,t)$.
\end{theorem} 

The method employed to prove Theorem~\ref{thm:3} will be the same as was used in~\cite{NaSi} to prove 
$\tau_q(3t,t)=\sigma_q(2t+1,t)$. In particular, we will use Theorem~\ref{thm:1} in Section~\ref{sec:intro}.
We first introduce the relevant definitions and a useful Lemma due to Govaerts~\cite{Go}. 
A set of points $B$, i.e., 1-spaces of $V$, is called a {\em blocking set} with respect to the $t$-spaces of 
$V$ if $W\cap {B}\not=\{0\}$ for any $t$-space $W$ in $V$. Note that any $(n-t+1)$-dimensional subspace 
of $V$ is a blocking set with respect to the $t$-spaces of $V$. Such blocking sets are 
called {\em trivial}. The following lemma follows from the results of Govaerts (see Case~2, page~612 in~\cite{Go}). 
\begin{lemma}[Govaerts~\cite{Go}]\label{lem:Go2}
Let $n$ and $t>1$ be integers such that $n\geq 2t$. 
If $\S$ is a minimum size maximal partial $t$-spread of $V(n,q)$, 
then $\bigcup_{W\in \S}W$ contains a trivial blocking set.
\end{lemma}

In the proof of Theorem~\ref{thm:3} we will also use the following proposition.

\begin{pro}\label{pro:2}
Let $d, d',$ and $n$ be integers such that $0 < d'< d \leq n/2.$ Then
\[
\sigma_q(n,d)<\sigma_q(n,d')\;.
\] 
\end{pro}
\begin{proof} We will prove that $\sigma_q(n,t)<\sigma_q(n,t-1)$ holds, for $1<t\leq n/2$.

If $t$ divides $n$, then $\sigma_q(n,t)=\theta_n/\theta_t$. Consequently, by Theorem \ref{thm:1} 
and with the use of Equation \eqref{eq:4}, we note that it is always true that
\[
\frac{\theta_n}{\theta_t}\leq \sigma_q(n,t)<\frac{\theta_n}{\theta_{t}}+q^\beta\;,
\]
where $0\leq r=n-kt<t$ and $\beta=\lceil (t+r)/2\rceil$. 
As $\theta_t>q\theta_{t-1}$ and $q^\beta<\theta_n/\theta_t$, we thus get
\[
\sigma_q(n,t) <2\frac{\theta_n}{\theta_{t}}\leq q\frac{\theta_n}{\theta_{t}}<\frac{\theta_n}{\theta_{t-1}}\leq\sigma_q(n,t-1)\;.
\] 
\end{proof}
\begin{proof}[{\bf Theorem~\ref{thm:3}}]
By Lemma~\ref{lem:Go}, we have  $\tau_q(n+t-1,t)\leq \sigma_q(n,t)$.
So, it remains to show that 
\begin{equation}\label{eq:lb}
\tau_q(n+t-1,t)\geq \sigma_q(n,t).
\end{equation}
Let $\S$ be a minimum size maximal partial $t$-spread in $V(n+t-1,q)$. Then by 
Lemma~\ref{lem:Go2}, $A=\bigcup_{W\in \S}W$ contains a trivial blocking set. In other words, 
there exists an $n$-dimensional subspace ${B}\subseteq A$.
Let 
\[\Pi_S=\{W\cap {B}:\; W\in \S\}.\] 
Since $B$ is a blocking set with respect to $t$-spaces, we have $W\cap {B}\not=\{0\}$
for any $W\in\S$. Thus, $\Pi_\S$ is a subspace partition of ${B}\cong V(n,q)$ containing 
subspaces of dimensions at most $t$. If $\Pi_\S$ contains a $t$-subspace,
then it follows from Theorem~\ref{thm:1} and the minimality of $\S$ that 
\[\tau_q(n+t-1,t)=|\S|=|\Pi_\S|\geq \sigma_q(n,t).\]
If $\Pi_\S$ does not contain any $t$-subspace, then each subspace in $\Pi_\S$
has dimension at most $t-1$ (and contains at most $\th_{t-1}$ $1$-dimensional subspaces). 
So the theorem now follows from the fact that the function $\sigma_q(n,t)$ is antimonotone in 
$t$ by Proposition \ref{pro:2}.
\end{proof}

\section{Some remarks}

Let $\Pi$ be a subspace partition of $V=V(n,q)$ consisting of $n_i$ subspaces of dimension $d_i$, for $1\leq i\leq k$. Let us assume that $d_1<d_2<\dots<d_k$ (and $n_1n_2\cdots n_k\neq0$). In \cite{He2} a lower bound on $n_1$ was given as a function of $q$, $d_1$ and $d_2$, and, as easily verified from that result, it is always true that $n_1\geq\sigma_q(d_2,d_1)$. 
Working on the results of this paper has given us many indications that the following conjecture holds.
\begin{conj}\label{conj:1}
Let $\Pi$ be a subspace partition of $V(n,q)$ with  $n_i>0$ subspaces of dimension $d_i$,
$1\leq i \leq k$, and where $d_1<\ldots<d_k$. Then, for any integer $j$, $1\leq j<k$, we have 
\[ n_1+\ldots+n_j\geq\sigma_q(d_{j+1},d_j).\]
\end{conj}

Let us also remark that for $n\leq 2t-1$, the problem
of determining the minimum size $\tau_q(n+t-1,t)$ of a maximal partial $t$-spread in $V(n+t-1,q)$ is still open. For $t=2$ and $n=3$,
the following lower bound was achieved by Glynn~\cite{Gl}:
\[\tau_q(4,2) \geq 2q,\]
while the following two upper bounds are due to G\'acs and Sz\"onyi~\cite{GaSz}: 
\[
 \tau_q(4,2)\leq  (2\log_2 q+1)q+1,\; \mbox{ if $q$ odd},\]
 and 
 \[ \tau_q(4,2)\leq (6.1\ln q+1)q+1,\; \mbox{ if $q>q_0$ even.}\]
\bs

\n{\bf Acknowledgment.}
The last three authors of this paper wish to thank their host Olof Heden and the Department of Mathematics at KTH for their warm hospitality while working on this paper during their visit there. 



\bs\n O. Heden (olohed@math.kth.se), Department of Mathematics, KTH, S-100 44 Stockholm, Sweden.

\bs\n J. Lehmann  (jlehmann@math.uni-bremen.de),  
Department of Mathematics, Bremen University, Bibliothekstrasse 1 - MZH, 28359 Bremen, Germany.

\bs\n E. N\u{a}stase (nastasee@xavier.edu): Department of Mathematics and Computer Science, Xavier University,  3800 Victory Parkway, Cincinnati, Ohio 45207.

\bs\n P. Sissokho (psissok@ilstu.edu): Mathematics Department, Illinois State University, Normal, Illinois 61790.
\end{document}